\documentclass[12pt,a4paper]{article}
\usepackage{graphicx}
\usepackage[]{amsmath,amssymb,amsthm}
\usepackage{paralist}
\usepackage{epstopdf}

\newtheorem{theorem}{Theorem}[section]
\newtheorem{proposition}[theorem]{Proposition}
\newtheorem{lemma}[theorem]{Lemma}

\theoremstyle{definition}
\newtheorem{definition}[theorem]{Definition}
\theoremstyle{remark}

\newcommand{\R}{\mathbb{R}}
\newcommand{\N}{\mathbb{N}}

\begin{document}

\begin{center}
{\large{\bf Learning by replicator and best-response: the importance of being indifferent}}\\
\mbox{} \\
\begin{tabular}{c}
{\bf Sofia B.\ S.\ D.\ Castro$^{\dagger}$} \\
{\small sdcastro@fep.up.pt} 
\end{tabular}

\end{center}

\noindent $^{\dagger}$ Faculdade de Economia and Centro de Matem\'atica, Universidade do Porto, Rua Dr.\ Roberto Frias, 4200-464 Porto, Portugal; fax: +351 225 505 050; phone: +351 225 571 100.

\begin{abstract}
This paper compares two learning processes, namely those generated by replicator and best-response dynamics, from the point of view of the asymptotics of play. We base our study on the intersection of the basins of attraction of locally stable pure Nash equilibria for replicator and best-response dynamics. 
Local stability implies that the basin of attraction has positive measure but there are examples where the intersection of the basin of attraction for replicator and best-response dynamics is arbitrarily small.
We provide conditions, involving the existence of an unstable interior Nash equilibrium, for
the basins of attraction of {\em any} locally stable pure Nash equilibrium under replicator and best-response dynamics to intersect in a set of positive measure. Hence, for any choice of initial conditions in sets of positive measure, if a pure Nash equilibrium is locally stable, the outcome of learning under either procedure coincides.
We provide examples illustrating the above, including some for which the basins of attraction exactly coincide for both learning dynamics. We explore the role that indifference sets play in the coincidence of the basins of attraction of the stable Nash equilibria.
\end{abstract}

\noindent {\em Keywords:} best-response dynamics; replicator dynamics; learning; basin of attraction

\noindent {\em JEL code:} C73
\vspace{.3cm}

\paragraph{Acknowledgements:}
I am grateful to S.\ van Strien for stimulating conversations. These took place during a visit of mine to Imperial College London, whose hospitality is gratefully acknowledged.

Many thanks also to J.\ Hofbauer for his insightful comments on an earlier version of this paper, and to J.\ Gaspar for help with the numerical simulations.

This research was partly supported by Centro de Matem\'atica da Universidade do Porto (UID/MAT/00144/2013), funded by the Portuguese Government through the Funda\c{c}\~ao para a Ci\^encia e a Tecnologia with national (Minist\'erio da Educa\c{c}\~ao e Ci\^encia) and European structural funds through the programs FEDER, under the partnership agreement PT2020, as well as by a grant from the Reitoria da Universidade do Porto.

\section{Introduction}
Learning in games has been drawing researchers' attention from many viewpoints in the hope to predict play. This prediction depends naturally upon the learning procedure and hence, it becomes interesting to understand how different learning mechanisms compare to one another. 
In particular, it is interesting to understand how predictions are robust to different learning specifications.
We address this question, in a simple continuous time setting, for models where there is a choice of Nash equilibria by comparing the set of points that converge to a given Nash equilibrium under two distinct learning processes. 

From an experimental perspective, Erev and Roth (1998) and Roth and Erev (1995)\footnote{The bibliographic references presented in this introduction are merely an illustration of the points made. They do not intend to be comprehensive in any way. Several survey articles exist already that serve this purpose. See, for instance, Hofbauer (2011), Hofbauer and Sigmund (2003) or Sandholm (2012).} 
try to establish which model of learning best fits the learning mechanisms displayed by subjects in experiments. The options are reinforcement learning or stochastic fictitious play (or a combination of both). Considering discrete time modelling, Hopkins (2002) stresses the similarities between these two learning procedures by showing that, asymptotically, reinforcement learning can generate the same result as stochastic fictitious play. Convergence properties of reinforcement learning have been established by Beggs (2005), whereas Mengel (2012) shows that a reinforcement learning across many games can destabilise strict Nash equilibria. 

When more than one equilibrium is available, the question of choice among the Nash equilibria arises. Duffy and Hopkins (2005) show that there is convergence to a pure Nash equilibrium under both reinforcement learning and stochastic fictitious play in market entry games.

The present work focusses on the two classic learning procedures of (con\-ti\-nuous-time) {\em best-response} (and, implicitly, its time-scaled fictitious play) and {\em replicator}, namely from the point of view of convergence properties. See Brown (1949), 
Matsui (1992)
and Hofbauer and Sigmund (2003) 
for recalling the notions of fictitious play, best-response and replicator dynamics, respectively. This study is particularly relevant when more than one pure Nash equilibrium is available.
It is well-known that Nash equilibria for these learning procedures are the same.  
Existence, and even local stability, is clearly not sufficient for predicting play since the choice of play towards one or another equilibrium greatly depends on initial conditions. The choices made from a given initial condition depend upon the basin of attraction in which they are found.
Zhang and Hofbauer (2015) 
discuss equilibrium selection in a $2 \times 2$ coordination game under replicator dynamics. One of their selection methods compares the size of the basin of attraction of various equilibria.
However, except when a unique equilibrium is globally stable (see Hofbauer and Sandholm (2009)
for examples in the context of stable games),
the basins of attraction of equilibria need not coincide for replicator and best-response dynamics. 
Golman and Page (2010) 
construct a one-parameter family of $3 \times 3$ games for which the basins of attraction of one equilibrium, under the two learning rules (replicator and best-response), intersect in a set of vanishing measure. Hence, for practical purposes, the two learning procedures predict very different outcomes of play. In fact, almost all initial conditions that converge to a given equilibrium under replicator dynamics, do not do so under best-response dynamics. This result persists in the context of aggregate behaviour of populations, see Golman (2011). 

When more than one pure Nash equilibrium is available as a possible outcome, in the sense that it is locally stable, we are interested in comparing the set of points from which an equilibrium is chosen under the two learning procedures of replicator and best-response. 
We show that the existence of an unstable interior Nash equilibrium, together with an invariance assumption,
guarantees that the intersection of the basins of attraction of {\em any} locally stable pure Nash equilibrium under replicator and best-response dynamics is not arbitrarily small. Hence, with non-vanishing probability, there exist initial conditions for which learning under replicator produces the same outcome as learning under best-response.
Our results thus offer insight into the relevance of the existence of a fully mixed Nash equilibrium. The existence of a fully mixed Nash equilibrium guarantees that there is at least one point in state space for which a player is indifferent among all actions, or equivalently, every type is present in the corresponding mix for the population. 

The size of the sets of initial conditions from which learning produces the same outcome under replicator and best-response dynamics is related to the existence of invariant lines of states at which the player in indifferent between exactly two strategies.
We provide examples illustrating both applications of our result and the importance of our hypothesis. The latter include a family of games constructed in Golman and Page (2010) and a new family of games created for this purpose.
The examples illustrative of our result are obtained from Zeeman's (1980) 
classification of replicator dynamics for $3 \times 3$ games. The examples, as developed here, address an open question left by Zhang and Hofbauer (2015), 
namely, that of comparing basin dominance of equilibria under different learning dynamics. 

It should be clear that when the basins of attraction for replicator and best-response dynamics intersect in a set of positive measure, play need not proceed in the same way, nor need it produce the same asymptotic behaviour, for both dynamics and from {\em all} initial conditions. We show examples where attraction properties of play coincide exactly, that is, for each initial condition the same Nash equilibrium is chosen under replicator and best-response dynamics. These examples exhibit invariant spaces for both types of dynamics which divide the  state space into invariant regions. Generically, saying that basins of attraction intersect in a set of positive measure is to say that, with positive probability, there are initial conditions from which players choose converging actions. Our examples suggest that the intersection of basins of attraction of pure Nash equilibria in the presence of a fully mixed Nash equilibrium is a large subset of state space. Therefore, play under replicator and best-response dynamics produces the same outcome with high probability.

The next section details the preliminary notions and results required for reading this article. Section~\ref{sec:equivalence} shows that if a fully mixed unstable Nash equilibrium exists, and an easily verifiable invariance condition holds, then learning by replicator and best-response dynamics produces the same outcome for a set of initial conditions that is non-vanishing. This section also establishes some additional results concerning the two learning mechanisms. In particular, it provides a sufficient condition for the exact coincidence of the basins of attraction of a pure Nash equilibrium under both replicator and best-response dynamics.
Section~\ref{sec:examples} provides some illustrative examples in $3 \times 3$ games. The final section concludes.

\section{Preliminaries}

As usual in population dynamics, we assume that the population consists of $n$ different types and denote by $x_i$, $i=1, \hdots, n$ the frequency of each type. It is clear that these frequencies must sum to one. Variations in the frequencies depend on how the fitness of each frequency compares to the average fitness of the population, when learning occurs by replicator dynamics. Under best-response dynamics, the update of the frequencies is made by choosing a best-response to the current mix in the population. 

Denote by $\Delta $ the set of frequency/probability vectors, that is,
$$
\Delta = \{ x \in \R^n: \; \; x_1+ \hdots + x_n =1; \;\; x_i \geq 0 \}.
$$
This is the natural state space for the game dynamics we consider.
Let $\partial \Delta$ denote the boundary of $\Delta$
and $int(\Delta)$ its interior.

Consider a game described by an $n \times n$ matrix $A$. If the game is played according to {\em replicator dynamics} (henceforth, RD) then each $x_i \in [0,1]$ evolves according to the following rule
\begin{equation}\label{eq:RD}
\dot{x}_i = x_i \left( (Ax)_i-x. Ax \right), \;\; i=1, \hdots, n.
\end{equation}
If, on the other hand, the game is played according to {\em best-response dynamics} (henceforth, BRD) then each state variable follows
\begin{equation}\label{eq:BRD}
\dot{x}_i \in \mbox{BR}(x_i)-x_i, \;\; i=1, \hdots, n.
\end{equation}
We denote by $e_i$, $i=1, \hdots, n$ the unit vectors in $\Delta$ and write BR$(x)=e_i$ when a best-response to $x$ is the choice of action $i$. We define, for $i=1, \hdots, n$, the set
$$
BR_i = \{ x \in \Delta : \; \; BR(x)=e_i \}.
$$

Even though BRD is a differential inclusion, it is really a differential equation except for those points in the {\em indifference sets}
\begin{equation}\label{eq:indifference}
Z_{i,j} = \{ x \in \Delta : \; \; (Ax)_i=(Ax)_j\}, \;\; i\neq j \in \{1, \hdots, n\}.
\end{equation}
It has been shown (see, for instance, Zeeman (1980)) 
that $A$ can be chosen so that $a_{ii}=0$ for all $i=1, \hdots, n$, which we assume to be the case from now on.
Notice that, since the diagonal of $A$ has only zeros, $(Ax)_i$ does not depend on $x_i$ for all $i=1, \hdots, n$. It is also worthwhile to mention that if $e_k \in Z_{i,j}$ for some $k=1, \hdots , n$ then the equation for $Z_{i,j}$ does not depend on $x_k$. This is a consequence of the fact that the equation for $Z_{i,j}$ is homogeneous as follows:

\begin{lemma}\label{lem:homogeneous}
If $e_k \in Z_{i,j}$ for some $k=1, \hdots , n$ then the equation for $Z_{i,j}$ does not depend on $x_k$.
\end{lemma}

\begin{proof}
The equation for $Z_{i,j}$ is of the form
$$
\sum_{i=1}^n a_i x_i = 0,
$$
where the real coefficients $a_i$ depend on the entries of $A$. The point $e_k$ written in coordinates is such that $x_k=1$ and $x_i=0$ for $i \neq k$. If $e_k \in Z_{i,j}$ then $x_k=1$ and $x_i=0$ for $i \neq k$ must solve the equation for $Z_{i,j}$. Direct substitution yields $a_k=0$.
\end{proof}

Nash equilibria are singularities of the dynamics and coincide under RD and BRD. If it exists, an interior or fully mixed Nash equilibrium occurs at the simultaneous intersection of all the indifference sets. We say that a Nash equilibrium, $x$, is {\em locally stable} if there exists an open neighbourhood, containing the Nash equilibrium, so that the orbits of points in this neighbourhood are attracted to the Nash equilibrium without leaving a possibly bigger neighbourhood of $x$. In the language of dynamical systems, the Nash equilibrium is locally asymptotically stable. The set of all points whose orbit is attracted to the Nash equilibrium is called its {\em basin of attraction}. We distinguish between the basins of attraction of the same Nash equilibrium, $x$, for different dynamics by writing ${\cal B}_{BRD}(x)$ and ${\cal B}_{RD}(x)$ for the basins of attraction under BRD and RD, respectively.

There are however points that are singularities of RD but are not Nash equilibria. According to the definition of a singularity for a differential equation, these are points for which the right-hand side of (\ref{eq:RD}) vanishes, which occurs either because $x_i=0$ or because $(Ax)_i=x. Ax$ if $x_i \neq 0$. This motivates the following definition of a singularity or stationary point for BRD. With this definition, singularities of RD and BRD coincide, except possibly at the vertices of $\Delta$. Recall that while $\partial \Delta$ is invariant for RD, that need not be the case for BRD.

\begin{definition}\label{singularity}
A point $x \in \partial \Delta$ is a singularity or stationary point for BRD if and only if $x \in Z_{i,j} \cap \{ x_k=0: \; \mbox{for all } \; k \neq i,j \}$.
\end{definition}

We assume throughout that 
\paragraph{Assumption A:} $Z_{i,j} \neq Z_{k,l}$ for all $\{i,j\} \neq \{k,l\}$.  
\medskip

This is a non-degeneracy condition ensuring that a player is not indifferent to more than two actions at a time, unless it is at a Nash equilibrium. 

\section{Equivalence of learning outcomes}\label{sec:equivalence}

We use the expression ``equivalence of learning outcomes'' as a way of expressing the fact that there is a non-vanishing set of initial conditions from which RD and BRD converge to the same Nash equilibrium. 

We start by pointing out some features of the two types of dynamics, mostly of geometric nature. Denote by $S_i$ the smallest open sector bounded by indifference sets and containing $e_i$. Note that the best-response is constant in $S_i$. When $e_i$ is a locally stable Nash equilibrium, we have $BR(x)=e_i$ for all $x \in S_i$ and one of the indifference sets in the boundary of $S_i$ is of the form $Z_{i,j}$. In this case, there are no invariant sets for BRD in $S_i$ and $S_i \subseteq BR_i$. Given Proposition 5.1 in Hofbauer {\em et al.} (2009), stating that the time-average of orbits of RD must converge to an invariant set under BRD, time-averages of orbits for RD of points in $S_i$ converge to $\partial S_i$, since there are no sets invariant for BRD in $S_i$.

Recalling Definition~\ref{singularity}, it is clear that the intersection of the boundary of $S_i$ and that of $\Delta$ contains no other singularities, for RD or BRD, than $e_i$ and, eventually, any singularities in $Z_{i,j} \cap \{ x_k=0: \; k \neq i,j \}$.

The set $S_i$ is non-empty and, under Assumption A, it is a set of positive measure. Its measure is determined by the space between the indifference sets of its boundary.

In Figure~\ref{fig:set_S_i}, we illustrate the set $S_i$, for $i=2$, for the game depicted in Figure~\ref{C6_2}. We choose this game because the indifference sets are generic. The game belongs to class $6_2$ of the classification in Zeeman (1980), see Section~\ref{sec:examples} for more detail.
The set $S_2$ is the open sector whose boundary is that of $\Delta$ together with the segments of $Z_{1,3}$ and $Z_{1,2}$ connecting $\partial \Delta$ to the interior Nash equilibrium. In this case, $S_2=S_1$ but in what follows we use $S_i$ only when $e_i$ is a Nash equilibrium. Hence, the choice $i=2$.

\begin{figure}[!htb]
\centerline{\includegraphics[width=0.5\textwidth]{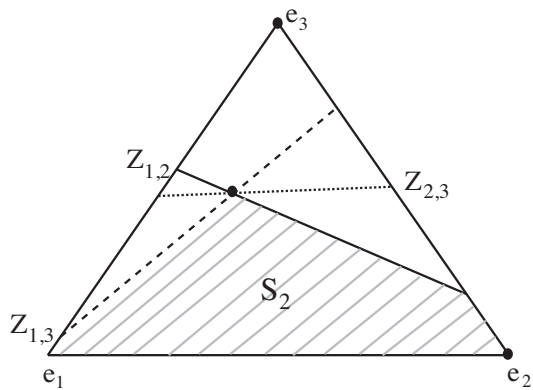}}
\caption{\small{The set $S_2$ for a game in class $6_2$ of Zeeman (1980). Full dots indicate Nash equilibria. The set $S_2$ is the open sector whose boundary is that of $\Delta$ together with the segments of $Z_{1,3}$ and $Z_{1,2}$ connecting $\partial \Delta$ to the interior Nash equilibrium. In this case, $S_2=S_1$.}
\label{fig:set_S_i}}
\end{figure}

The next result establishes conditions that guarantee equivalence of learning outcomes in a game.

\begin{theorem}\label{th:basins}
Consider a game with a pure locally stable Nash equilibrium, $e_i$, for which Assumption A holds and such that
\begin{itemize}
	\item[(H1)]  there exists an unstable fully mixed Nash equilibrium, $x^*$;
	\item[(H2)]  $S_i$ is invariant\footnote{The invariance of $S_i$ is readily checked by looking at the best-response on either side of its boundaries. Using Figure~\ref{fig:set_S_i} as an illustration, if between the boundary of $S_2$ consisting of $Z_{1,2}$ and the part of $Z_{2,3}$ immediately above it the best-response is $e_1$ then $S_2$ is invariant. If, on the other hand, the best-response for the same set of points is $e_3$ then $S_2$ is not invariant.} for RD.
\end{itemize}
Then, ${\cal B}_{RD}(e_i) \cap {\cal B}_{BRD}(e_i)$ is a set of  positive measure containing $S_i$. 
\end{theorem}

\begin{proof}
The existence of a fully mixed Nash equilibrium in Hypothesis (H1) and Assumption A ensure the existence of points arbitrarily close to $x^*$ which belong to $S_i \subseteq BR_i$. Let $x$ be such a point. Then $BR(x)=e_i$ and 
$$
\frac{d}{dt}\left(\frac{x_i}{x_j}\right)=\frac{x_i}{x_j} \left[ (Ax)_i-(Ax)_j \right] >0
$$
since $x \in BR_i$ implies $(Ax)_i>(Ax)_j$ for all $j \neq i$. Hence, the ratio $x_i(t)/x_j(t)$ increases with $t$ for all $j \neq i$. The invariance condition in Hypothesis (H2) then implies that the orbit of $x$ under RD remains in $S_i$. 

In view of Proposition 5.1 in Hofbauer {\em et al.} (2009) and the local stability of $e_i$ for both dynamics, the time-average of the orbit of $x$ under RD must converge to $e_i \in \partial \Delta$, and so does the orbit of $x$ under RD. 
In fact, Proposition 5.1 establishes not just the invariance of the limit of time-averages of orbits of RD but also that this limit must be internally chain transitive\footnote{According to Hofbauer {\em et al.} (2009), ``a set $A$ is internally chain transitive if any two points $x, y \in A$ can be connected by finitely many arbitrarily long pieces of orbits lying completely within $A$ with arbitrarily small jumps between them''. A precise definition can be found in Bena\"{\i}m et al. (2005).}. Chain-transitivity, together with the fact that the ratio $x_i(t)/x_j(t)$ is increasing in $S_i$, excludes the possibility of the limit of the time-average of RD being the whole of $\partial \Delta$. Because for points in the intersection of the closure of $S_i$ with $\partial \Delta$ the best-response is still $e_i$, there are no other attractors and the time-average of RD converges to $e_i$. Since the time-average of RD converges to a pure equilibrium, then so does the trajectory itself.

Note that convergence to $e_i$ of the orbit of $x$ under RD holds for all $x \in S_i$. Hence, $S_i \subset {\cal B}_{RD}(e_i) \cap {\cal B}_{BRD}(e_i)$.

The fact that, because of Assumption A, the measure of $S_i$ is positive finishes the proof.
\end{proof}

Equivalence of learning outcomes in one-parameter families of games follows from Theorem~\ref{th:basins}. It suffices to notice that both the existence of an interior Nash equilibrium and the invariance of $S_i$ are robust under perturbation of a game. 

Recall that an interior Nash equilibrium is fully mixed in the sense that every type is present in the corresponding mix for the population, that is, every type has positive frequency. Continuity of play ensures that given any type there exist points near the Nash equilibrium for which this type has positive frequency. Then, for a Nash equilibrium, there exist initial conditions for which the corresponding type is present. If this type determines a Nash equilibrium, its basins of attraction under the two different dynamics intersect in a non-vanishing set. 
Recall that when a given type is a best-response, the frequency of this type increases under both dynamics. 
The size of ${\cal B}_{RD}(e_i) \cap {\cal B}_{BRD}(e_i)$ is bounded below by the size of $S_i$ which depends on the relative position of the indifference sets that constitute its boundary.

We note that when the pure Nash equilibrium is a uniformly ESS, a stronger notion used by Golman and Page (2010) in their Erratum\footnote{The definition of uniformly ESS may be found in page 73 of Golman and Page (2010).}, we can use their Theorem 2 to obtain a much shorter proof.
This theorem states that a pure Nash equilibrium which is uniformly ESS and whose action $a$ is a best-response for a set of points of at most measure zero, has basins of attraction for RD and BRD with vanishing intersection. Such an action is said to have the {\em Never an Initial Best Response Property}. Equivalently, the result can be stated as follows: let $a$ be an action such that $a=$BR$(x)$ for $x$ in a set of positive measure; then, if $x^*$ is a pure  Nash equilibrium, uniformly ESS, corresponding to $a$, the intersection of the basins of attraction of $x^*$ for RD and BRD is not arbitrarily small. Note that for games satisfying Theorem~\ref{th:basins} the action $a$ in Golman and Page (2010) does not satisfy the {\em Never an Initial Best Response Property} and therefore the basins of attraction for the two dynamics intersect in a set of positive measure.

\subsection{Invariant indifference sets}

When some indifference sets are invariant for either or both learning mechanisms, further information on the equivalence of learning outcomes can be obtained. We start with a sufficient condition for invariance.
The following result generalizes that proved by Ochea (2010) in Chapter 2, Lemma 3 
in the particular case of a $3 \times 3$ coordination game under RD.

\begin{lemma}\label{lem:invariant}
Let $Z_{i,j}$ be an indifference set of an $n \times n$ game. If $e_k \in Z_{i,j}$ for all $k \neq i,j$ then $Z_{i,j}$ is invariant under RD.
\end{lemma}

\begin{proof}
Let $x \in Z_{i,j}$. According to equation (\ref{eq:indifference}), we have $(Ax)_i=(Ax)_j$ and therefore,
$$
\frac{\dot{x}_i}{x_i}=\frac{\dot{x}_j}{x_j}.
$$
Integrating with respect to $t$, we obtain for all $t$
$$
\log{x_i(t)}-\log{x_i(0)} = \log{x_j(t)}-\log{x_j(0)} \Leftrightarrow x_i(t) = \frac{x_i(0)}{x_j(0)} x_j(t).	
$$
Since $e_k \in Z_{i,j}$ for all $k \neq i,j$, the equation for $Z_{i,j}$ is of the form $x_i/x_j=K \in \R$ (see Lemma~\ref{lem:homogeneous}). The fact that $x \in Z_{i,j}$ finishes the proof.
\end{proof}

As a consequence, we establish a sufficient condition which prevents the occurrence of cyclic behaviour in $n \times n$ games. When cyclic behaviour occurs the frequency of each type increases and decreases in turn. See Figure~\ref{C5_1} (left) for an illustration.

\begin{proposition}\label{prop:no_cyclic}
In a $n \times n$ game for which there exist $n-2$ indifference sets such that $e_k \in Z_{i,j}$ for $k \neq i,j$, RD exhibits no cyclic behaviour.
\end{proposition}

\begin{proof}
Cyclic behaviour is dependent on the existence of an interior Nash equilibrium at which the Jacobian matrix has at least one pair of complex eigenvalues. Since, by Lemma~\ref{lem:invariant}, such a $Z_{i,j}$ is invariant for RD, it is the eigenspace for some eigenvalue. This ensures the existence of $n-2$ real eigenvalues. Since the state space $\Delta$ is $(n-1)$-dimensional, all eigenvalues must be real.
\end{proof}

Note that $Z_{i,j}$ is usually not invariant for RD. It is invariant for BRD generically only at points $x$ for which BR$(x)=e_i$ and BR$(x)=e_j$. However, under the hypothesis of Proposition~\ref{prop:no_cyclic}, $Z_{i,j}$ becomes invariant for BRD as well. This is because at the points where it is usually not invariant, the best-response now contains $e_k$ and $e_k \in Z_{i,j}$ for $k \neq i,j$.

Proposition~\ref{prop:no_cyclic}, when applied to $3 \times 3$ games, contributes with a correction to Zeeman's diagram for class $6_1$ as we show in the next section. In the context of $3 \times 3$ games, Proposition~\ref{prop:no_cyclic} states that provided that one indifference set is invariant there is no cyclic behaviour.

Recall that $Z_{i,j}$ divides $\Delta$ into two connected components. The invariance of these connected components for RD follows from that of the set $Z_{i,j}$, provided no other singularities exist in these connected components. A condition leading to the invariance of the connected components is thus provided by Lemma~\ref{lem:invariant}. In this case, $e_i, e_j \notin Z_{i,j}$ may be Nash equilibria outside $Z_{i,j}$.

\begin{theorem}\label{th:invariance}
Let $e_i$ be a locally stable pure Nash equilibrium, assume that Assumption A holds, that $e_i \notin Z_{i,j}$  and that the connected components of $\Delta \backslash Z_{i,j}$ are invariant for RD and BRD. Then ${\cal B}_{BRD}(e_i)={\cal B}_{RD}(e_i)$.
\end{theorem}

\begin{proof}
Since $Z_{i,j}$ divides $\Delta$ into two invariant connected components and $e_i \notin Z_{i,j}$, $e_i$ belongs to one of the invariant connected components. For $x$ in the connected component containing $e_i$ it is $BR(x)=e_i$ and hence, there are no other invariant sets in this connected component. Then, the orbit of $x$ under RD also converges to $e_i$ (as in the proof of Theorem~\ref{th:basins}), finishing the proof.
\end{proof}

Even though Theorem~\ref{th:invariance} is not stated for non-pure Nash equilibria on the boundary,  the cases illustrated by Figures 7 and 8 show that an analogous result may apply. This is beyond the scope of the present article as the interior of the boundary can accommodate complicated dynamics around a Nash equilibrium in higher dimensions.

\section{Examples}\label{sec:examples}

We illustrate our results by looking at all possible $3 \times 3$ games. These games have been divided by Zeeman (1980) into $19$ classes. Not all of these classes correspond to games with one fully mixed Nash equilibrium and at least one pure Nash equilibrium. In fact, the games belonging to classes $2$, $3$, $-4_2$, $-6_3$, $-6_4$, $7_3$, $8$, $-9_2$ and $-10_2$ do not have a fully mixed Nash equilibrium, and the games in class $1$ do not have pure Nash equilibria, see Figure 11 in Zeeman (1980). We illustrate the equivalence of learning outcomes for the remaining eight classes, thus completely addressing all $3 \times 3$ games where establishing equivalence of learning outcomes (or lack thereof) makes sense.

Since we are interested in illustrating learning procedures that lead to a pure Nash equilibrium, we use the symmetric matrix to that considered in Zeeman (1980) 
whenever we want to change the stability of the interior Nash equilibrium from stable to unstable. We note that Zeeman's classification is robust in the sense that small perturbations of the matrices defining the game lead to qualitatively equivalent dynamics. The symmetric of a matrix is obtained by multiplication of all entries by $(-1)$. This interchanges the stability of the singularities in the game. By presenting here the diagrams for BRD, we illustrate how large the intersection of the basins of attraction for the two dynamics is. We restrict our attention to the cases relevant in the illustration of Theorems~\ref{th:basins} and \ref{th:invariance}.
In Table ~\ref{table1} we list the classes and the matrices we use for our illustration, preserving the order of Zeeman (1980). The line below each matrix indicates whether there is a sign reversal with respect to the matrix used by Zeeman (1980). 
We choose to reverse the sign of the elements in the matrix when this creates more stable pure equilibria than the original matrix, and the interior Nash equilibrium becomes unstable.

\begin{table}[h!]
\begin{center}
\begin{tabular}{|c|c|c|}
\hline 
Class & $5_1$ & $6_1$ \\
\hline 
Matrix $A$ & $\left(\begin{array}{ccc} 
0 & -3 & 1\\
-1 & 0 & -1 \\
-3 & 1 & 0 
\end{array} \right)$ & 
$\left(\begin{array}{ccc} 
0 & -1 & -1\\
1 & 0 & -3 \\
-1 & -1 & 0 
\end{array} \right)$ \\
\hline 
Sign reversal & yes & yes \\
\hline 
\hline 
Class & $7_1$ & $10_1$ \\
\hline 
Matrix $A$ & $\left(\begin{array}{ccc} 
0 & 6 & -4\\
-3 & 0 & 5 \\
-1 & 3 & 0 
\end{array} \right)$ &
$\left(\begin{array}{ccc} 
0 & -1 & -1\\
-1 & 0 & -1 \\
-1 & -1 & 0 
\end{array} \right)$ \\
\hline 
Sign reversal & no & yes \\
\hline 
\hline 
Class & $4_1$ & $6_2$ \\
\hline 
Matrix $A$ & $\left(\begin{array}{ccc} 
0 & -3 & 1\\
-3 & 0 & 1 \\
-1 & -1 & 0 
\end{array} \right)$ &
$\left(\begin{array}{ccc} 
0 & -1 & -3\\
1 & 0 & -5 \\
-1 & -3 & 0 
\end{array} \right)$ \\
\hline 
Sign reversal & yes & yes \\
\hline 
\hline 
Class & $7_2$ & $9_1$ \\
\hline 
Matrix $A$ & $\left(\begin{array}{ccc} 
0 & 1 & -1\\
-1 & 0 & 1 \\
-1 & 1 & 0 
\end{array} \right)$ &
$\left(\begin{array}{ccc} 
0 & -1 & 3\\
-1 & 0 & 3 \\
1 & 1 & 0 
\end{array} \right)$ \\
\hline 
Sign reversal & no & no \\
\hline 
\end{tabular}
\caption{Zeeman's classes and their matrices. We indicate the existence of a sign reversal with respect to the matrix originally used by Zeeman.\label{table1}}
\end{center}
\end{table}

A straightforward calculation shows that the indifference sets for BRD for each class in Table~\ref{table1} are given in Table~\ref{table2}. For classes $6_1$, $10_1$, $4_1$, $7_2$ and $9_1$ there is at least one indifference set which is invariant. The invariant indifference sets provide a bound for the basins of attraction of Nash equilibria for RD provided the Nash equilibrium is not on the indifference set. In this case, the basins of attraction coincide under both learning procedures and the question of basin dominance in Zhang and Hofbauer (2015) has the same answer in both replicator and best-response dynamics.

Although Proposition~\ref{prop:no_cyclic} is not an equivalence, we do observe the existence of cyclic behaviour in the games in the remaining classes. The cyclic behaviour in RD corresponds also to cyclic BRD.

\begin{table}[h!]
\begin{center}
\begin{tabular}{|c|c|c|}
\hline 
Class & $5_1$ & $6_1$ \\
\hline
$Z_{1,2}$ & $x_1-3x_2+2x_3=0$ & $x_1+x_2-2x_3=0$\\
\hline
$Z_{1,3}$ & $ 3x_1-4x_2+x_3=0$ & $x_1=x_3$ \\
\hline
$Z_{2,3}$ & $2x_1-x_2-x_3=0$ & $2x_1+x_2-3x_3=0$ \\
\hline
\hline
Class & $7_1$ & $10_1$ \\
\hline
$Z_{1,2}$  & $3x_1+6x_2-9x_3=0$ & $x_1=x_2$ \\
\hline
$Z_{1,3}$ & $x_1+3x_2-4x_3=0$ & $x_1=x_3$ \\
\hline
$Z_{2,3}$ & $2x_1+3x_2-5x_3=0$ & $x_2=x_3$ \\
\hline
\hline
Class & $4_1$ & $6_2$ \\
\hline
$Z_{1,2}$ & $x_1=x_2$ & $x_1+x_2-2x_3=0$ \\
\hline
$Z_{1,3}$ & $x_1-2x_2+x_3=0$ & $x_1+2x_2-3x_3=0$ \\
\hline
$Z_{2,3}$ & $2x_1-x_2-x_3=0$ & $2x_1+3x_2-5x_3=0$ \\
\hline
\hline
Class & $7_2$ & $9_1$ \\
\hline
$Z_{1,2}$ & $x_1+x_2-x_3=0$ & $x_1=x_2$ \\
\hline
$Z_{1,3}$ & $x_1=x_3$ & $x_1+2x_2-3x_3=0$ \\
\hline
$Z_{2,3}$ & $x_2=x_3$ & $2x_1+x_2-3x_3=0$ \\
\hline
\end{tabular}
\caption{The indifference sets for the classes of Table~\ref{table1} consist of points satisfying the equation in each line. \label{table2}}
\end{center}
\end{table}

In Figures~\ref{C5_1} -- \ref{C9_1}, we show the phase diagrams for RD and BRD for the above classes. The diagram for class $6_1$ for RD is here corrected as there is no cyclic behaviour. This is a consequence of Proposition~\ref{prop:no_cyclic} but can also be checked directly by computing the eigenvalues of the Jacobian matrix for RD at the interior Nash equilibrium.

\begin{figure}[!htb]
\centerline{\includegraphics[width=0.8\textwidth]{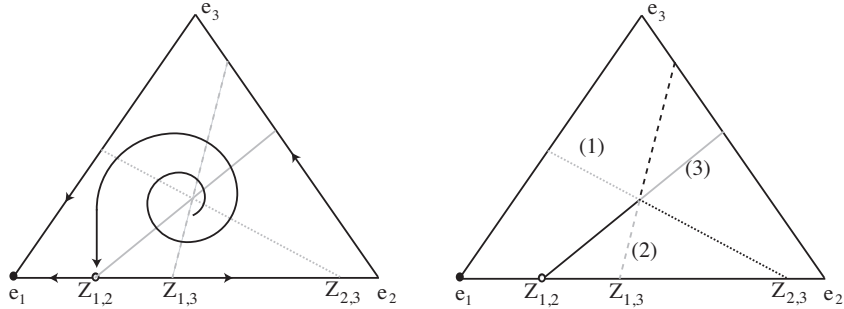}}
\caption{\small{RD (left) and BRD (right) for class $5_1$. There is only one stable Nash equilibrium, $e_1$. For RD, its basin of attraction is $\Delta$ except for the points on the stable manifold of $e_{1,2}$. For BRD, the only points that may not converge to $e_1$ are those on $Z_{1,2}$ below the interior Nash equilibrium.}
\label{C5_1}}
\end{figure}

\begin{figure}[!htb]
\centerline{\includegraphics[width=0.8\textwidth]{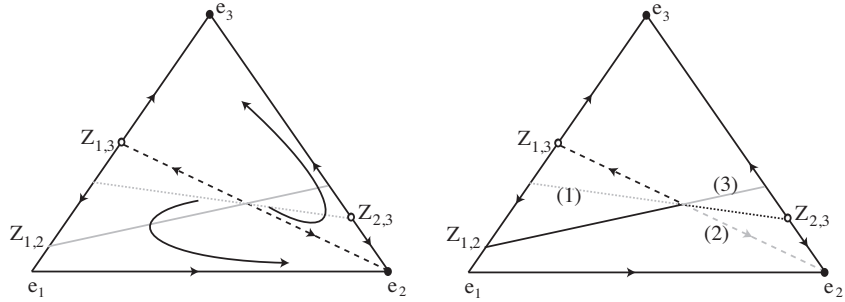}}
\caption{\small{RD (left) and BRD (right) for class $6_1$. There are two stable Nash equilibria, $e_2$ and $e_3$. For RD, the line $Z_{1,3}$ is invariant and constitutes the boundary of the basins of attraction of $e_2$ and $e_3$. For BRD, the boundary between the basins of attraction of $e_2$ and $e_3$ is the part of $Z_{1,3}$ above the interior Nash equilibrium and the part of $Z_{2,3}$ to the right of the interior Nash equilibrium. Hence, $e_2$ attracts more initial conditions under BRD than under RD.}
\label{C6_1}}
\end{figure}

\begin{figure}[!htb]
\centerline{\includegraphics[width=0.8\textwidth]{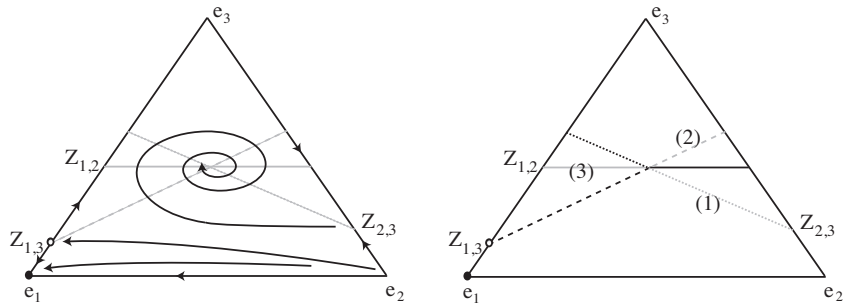}}
\caption{\small{RD (left) and BRD (right) for class $7_1$ . There are two stable Nash equilibria but only $e_1$ is pure. Its basin of attraction is much larger under BRD than under RD.}
\label{C7_1}}
\end{figure}

\begin{figure}[!htb]
\centerline{\includegraphics[width=0.8\textwidth]{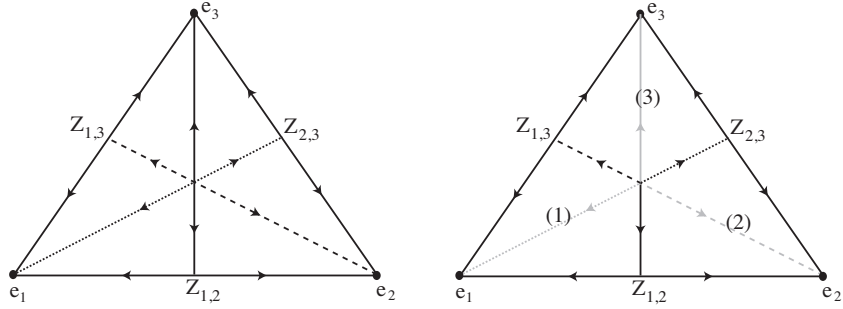}}
\caption{\small{RD (left) and BRD (right) for class $10_1$. All three pure strategies are Nash equilibria. Their basins of attraction totally coincide under both dynamics. This is because all three indifference sets are invariant and the boundary of the basins of attraction is made of parts of these indifference sets.}
\label{C10_1}}
\end{figure}

\begin{figure}[!htb]
\centerline{\includegraphics[width=0.8\textwidth]{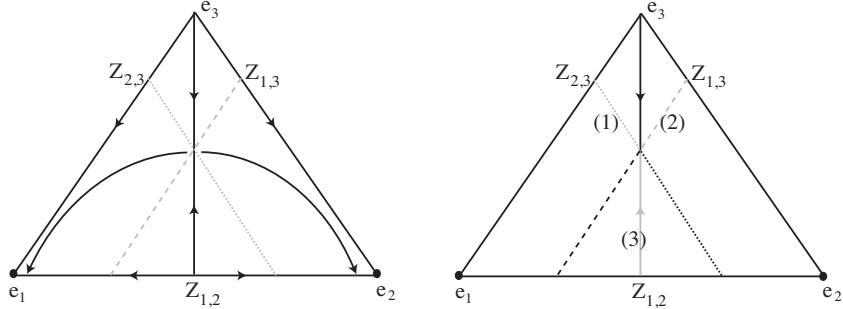}}
\caption{\small{RD (left) and BRD (right) for class $4_1$. The two stable Nash equilibria are $e_1$ and $e_2$. Their basins of attraction under RD and BRD coincide as they are divided by $Z_{1,2}$, which is the stable manifold of the interior Nash equilibrium.}
\label{C4_1}}
\end{figure}

\begin{figure}[!htb]
\centerline{\includegraphics[width=0.8\textwidth]{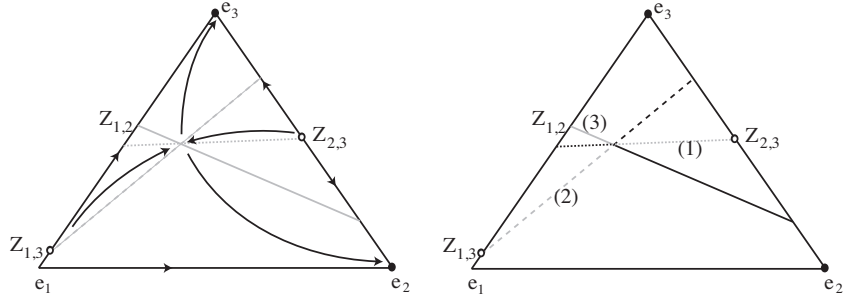}}
\caption{\small{RD (left) and BRD (right) for class $6_2$. There are two stable Nash equilibria, $e_2$ and $e_3$. Even though the intersection of the basins of attraction under RD and BRD is a set of positive measure, they are different. For RD, the stable manifold of the interior Nash equilibrium divides the basin of attraction of $e_2$ from that of $e_3$. For BRD, this division is made by $Z_{2,3}$ to the left of the interior Nash equilibrium and $Z_{1,3}$ to its right.}
\label{C6_2}}
\end{figure}

\begin{figure}[!htb]
\centerline{\includegraphics[width=0.8\textwidth]{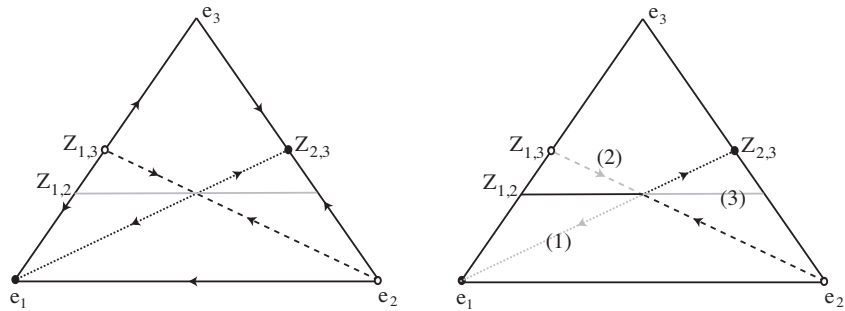}}
\caption{\small{RD (left) and BRD (right) for class $7_2$. There are two stable Nash equilibria but only $e_1$ is pure. The basins of attraction of $e_1$ and $e_{2,3}$ are the same under RD and BRD and are bounded by the invariant set $Z_{1,3}$.}
\label{C7_2}}
\end{figure}

\begin{figure}[!htb]
\centerline{\includegraphics[width=0.8\textwidth]{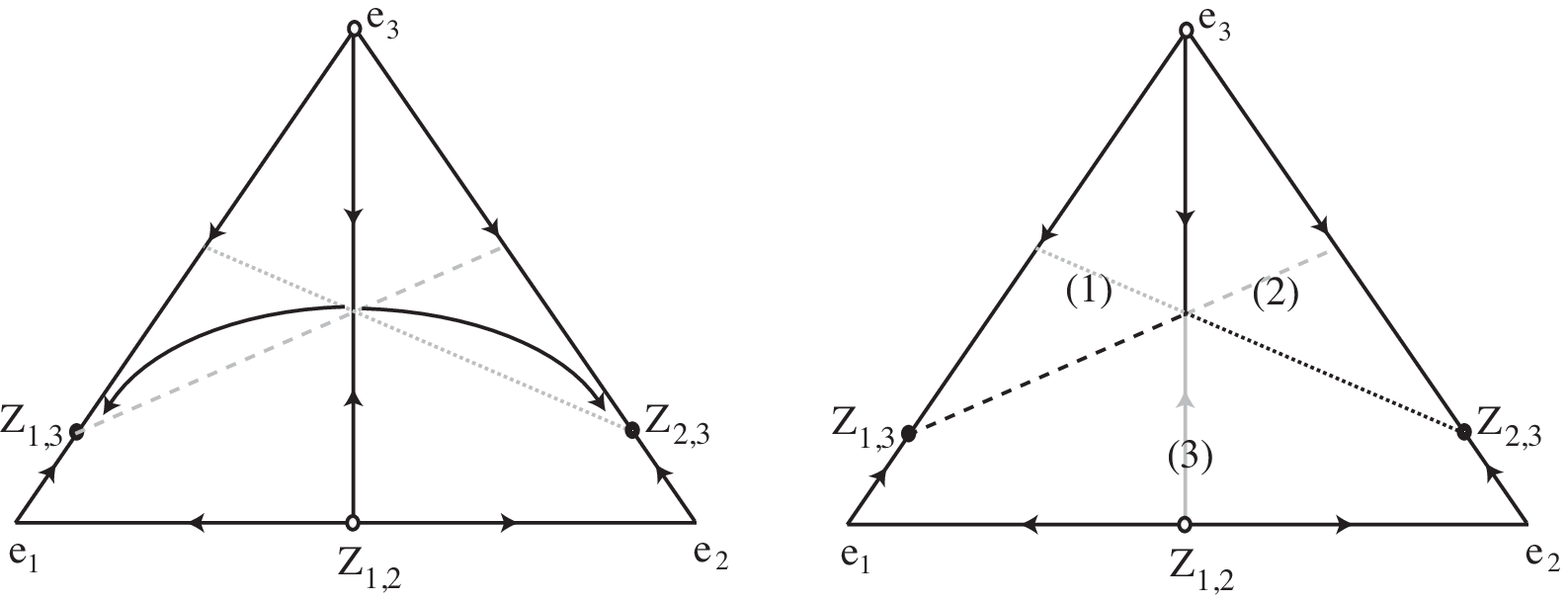}}
\caption{\small{RD (left) and BRD (right) for class $9_1$. The two stable equilibria are $e_{1,3}$ and $e_{2,3}$. Their basins of attraction under RD and BRD coincide and their boundary is the invariant indifference set $Z_{1,2}$.}
\label{C9_1}}
\end{figure}

In all the figures we use a full dot to indicate a stable Nash equilibrium and an open dot to indicate the Nash equilibrium is unstable. We do not mark the interior Nash equilibrium in order to preserve the clarity of the figures. It is easy to see this lies on the intersection of all the indifference lines. We use a full line for indicating $Z_{1,2}$, a dashed line for $Z_{1,3}$ and a dotted line for $Z_{2,3}$. When the indifference lines are invariant, we place arrows on them to indicate the direction of play. For RD, we present non-invariant indifference lines in grey. For BRD, we use grey when there is no change in the best-response along the indifference line. Also for BRD, we indicate in brackets the action which constitutes a best-response in the set of points bounded by black parts of indifference lines. As the boundary of $\Delta$ is invariant under RD, we indicate with arrows the direction of the flow on $\partial \Delta$. We denote by $e_{i,j}$ the Nash equilibrium in $Z_{i,j} \cap \{ x_k=0; \; k \neq i,j \}$.

The class in Figure~\ref{C5_1} does not satisfy Hypothesis (H2) in Theorem~\ref{th:basins}. The class in Figure~\ref{C7_1} does not satisfy any of the hypotheses in the same theorem. In both cases, the interior Nash equilibrium has complex eigenvalues. We note that the instability of the interior Nash equilibrium in Figure~\ref{C5_1} suffices to produce highly coincident basins of attraction. However, in Figure~\ref{C5_1}, not all points in $S_1$ belong to ${\cal B}_{RD}(e_1)$ since some converge to the intersection of $Z_{1,2}$ with $\partial \Delta$. We note that $S_1$ is not invariant for RD in this case.
The class in Figure~\ref{C7_1} is used in Subsection~\ref{sec:no_good} to construct a family of games with an interior Nash equilibrium for which ${\cal B}_{RD}(e_1)$ and ${\cal B}_{BRD}(e_1)$ have vanishing intersection.

Figures~\ref{C6_1}, \ref{C10_1}, \ref{C4_1}, \ref{C7_2} and \ref{C9_1} show invariant indifference sets. This is not sufficient to ensure that the basins of attraction exactly coincide under RD and BRD. However, these do coincide in Figures~\ref{C10_1}, \ref{C4_1}, \ref{C7_2} and \ref{C9_1}, where for each Nash equilibrium on the boundary there is an invariant indifference set that does not contain the Nash equilibrium, illustrating Theorem~\ref{th:invariance}.
Note that, in Figures~\ref{C4_1} and \ref{C9_1}, one invariant indifference line is enough to guarantee the coincidence of the basins of attraction for both dynamics. This is insufficient in Figure~\ref{C6_1}, where one pure Nash equilibrium belongs to the only invariant indifference set. In Figure~\ref{C10_1} all indifference lines are invariant and it is then clear that the basins of attraction must coincide.

Figure~\ref{C6_2} illustrates the game already used in Figure~\ref{fig:set_S_i}. Even though the intersection of the basins of attraction of $e_2$ and $e_3$ under RD and BRD is a set of positive measure, they are different. For RD, the stable manifold of the interior Nash equilibrium divides the basin of attraction of $e_2$ from that of $e_3$. For BRD, this division is made by $Z_{2,3}$ to the left of the interior Nash equilibrium and $Z_{1,3}$ to its right. Observe that $S_2 \subset {\cal B}_{RD}(e_2) \cap {\cal B}_{BRD}(e_2)$.

\bigbreak

We illustrate the convergence to equilibrium for classes $4_1$ and $6_2$ corresponding to Figures~\ref{C4_1} and \ref{C6_2}, respectively. In Figure~\ref{dynC4_1}, the BRD and RD are illustrated for class $4_1$. For BRD the figure depicts the trajectory from several distinct initial conditions whereas for RD we plot the vector field (using Mathematica). 
In Figure~\ref{dynC6_2}, we illustrate the same but for class $6_2$.

\begin{figure}[!htb]
\centerline{\includegraphics[width=0.8\textwidth]{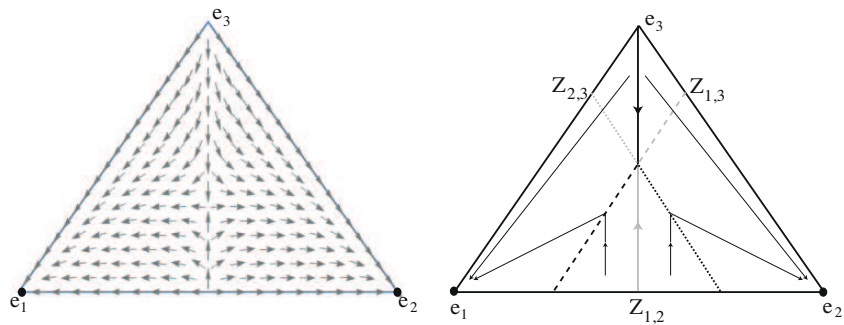}}
\caption{\small{RD (left) and BRD (right) for class $4_1$.}
\label{dynC4_1}}
\end{figure}

\begin{figure}[!htb]
\centerline{\includegraphics[width=0.8\textwidth]{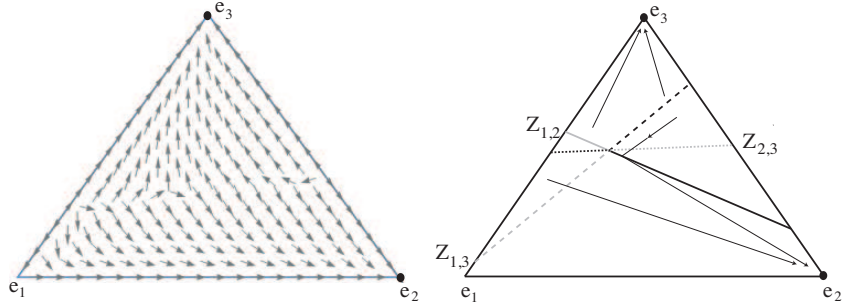}}
\caption{\small{RD (left) and BRD (right) for class $6_2$.}
\label{dynC6_2}}
\end{figure}

Recall that in class $4_1$ the basins of attraction of $e_1$ and $e_2$ coincide for both types of dynamics and this is clear from Figure~\ref{dynC4_1}. That the intersection of the basins of attraction for RD and BRD contains a set of positive measure for class $6_2$ is also clear from Figure~\ref{dynC6_2}.

\subsection{Basins of attraction with vanishing intersection}\label{sec:no_good}

We present two families of games for which a pure Nash equilibrium has basins of attraction for RD and BRD with vanishing intersection. The first example has no fully mixed Nash equilibrium and was presented by Golman and Page (2010). To illustrate the fact that the existence of a fully mixed Nash equilibrium is not sufficient for the non-vanishing intersection of the basins of attraction, we construct an example with a fully mixed Nash equilibrium and such that the basins of attraction for RD and BRD have vanishing intersection. This example consists in a family containing class $7_1$ for $n=1$. 

\paragraph{An example with no fully mixed Nash equilibrium:} Consider the example presented by Golman and Page (2010).
The matrix defining the family of games is
$$
\left( \begin{array}{ccc}
1 & -N & -N^{-1} \\
2-N^3 & 2 & 2 \\
0 & 0 & 0
\end{array} \right); \; \; \;  N >1.
$$
Subtracting $1$ from the first column and $2$ from the second, we obtain the equivalent family defined by
$$
\left( \begin{array}{ccc}
0 & -N-2 & -N^{-1} \\
1-N^3 & 0 & 2 \\
-1 & -2 & 0
\end{array} \right),
$$
which belongs to class $6$ of Zeeman's classification (Zeeman 1980). 
There are no invariant indifference lines and the intersection of the indifference lines occurs outside the simplex, that is, there is no interior Nash equilibrium. There are four Nash equilibria: $e_1$, $e_2$, $e_{1,2}$ and $e_{1,3}$. The pure Nash equilibria are locally stable.
The indifference lines depend on $N$ and are given by
\begin{eqnarray*}
Z_{1,2} & = & \{ x \in \Delta: \; \; (N^3-1)x_1-(N+2)x_2-\frac{2N+1}{N}x_3=0 \} \\
Z_{1,3} & = & \{ x \in \Delta: \; \; x_1-Nx_2-\frac{1}{N}x_3=0 \} \\
Z_{2,3} & = & \{ x \in \Delta: \; \; (2-N^3)x_1+2x_2+2x_3=0 \}.
\end{eqnarray*}
As $N \rightarrow +\infty$ the indifference lines $Z_{1,2}$ and $Z_{2,3}$ move closer to the side of $\Delta$ connecting $e_2$ to $e_3$ whereas $Z_{1,3}$ moves closer to the side of $\Delta$ connecting $e_1$ to $e_2$. As illustrated in Figure~\ref{fig:Golman}, for BRD the basins of attraction of $e_1$ and $e_2$ are divided by $Z_{1,3}$, while for RD these basins of attraction are divided by the invariant manifold connecting $e_{1,2}$ and $e_{1,3}$. As $N \rightarrow +\infty$, $Z_{1,3}$ moves so that, for BRD, the basin of attraction of $e_1$ shrinks while the invariant manifold connecting $e_{1,2}$ and $e_{1,3}$ moves so that, for RD, the same basin takes up almost all of $\Delta$. 

Comparing this with Figure~\ref{C6_2} helps understand the role of the interior Nash equilibrium. In Figure~\ref{C6_2}, the basins of attraction for RD and BRD differ as much as the invariant line $Z_{2,3}$ differs from an invariant manifold for RD. In this case, however, the invariant manifold connects the interior Nash equilibrium, which also belongs to $Z_{2,3}$, and therefore cannot be transformed away from $Z_{2,3}$. 
The next example shows the importance of this feature of the invariant manifold for the equivalence of learning outcomes.

\begin{figure}[!htb]
\centerline{\includegraphics[width=1\textwidth]{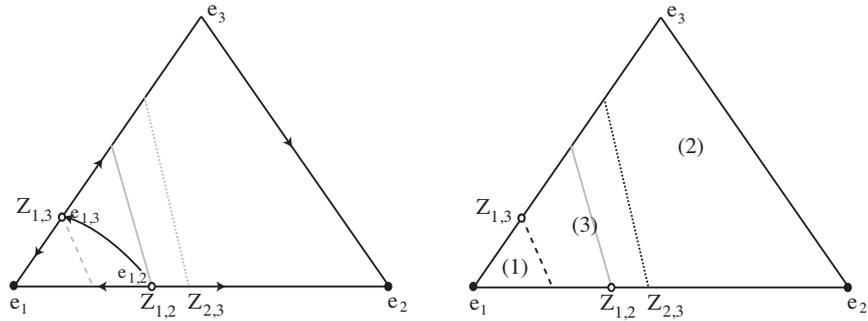}}
\caption{\small{RD (left) and BRD (right) for the family of games in (Golman and Page 2010). 
For RD the basins of attraction of $e_1$ and $e_2$ are divided by the invariant manifold connecting $e_{1,2}$ and $e_{1,3}$. For BRD the basins of attraction of $e_1$ and $e_2$ are divided by $Z_{1,3}$. As $N \rightarrow +\infty$, the invariant manifold moves towards $\partial \Delta$ along the line from $e_2$ to $e_3$, whereas $Z_{1,3}$ moves towards $\partial \Delta$ along the line from $e_1$ to $e_3$.}
\label{fig:Golman}}
\end{figure}

\paragraph{An example with a fully mixed Nash equilibrium:} Consider the learning mechanisms defined by the following family of matrices
$$
A_n=\left( \begin{array}{ccc}
0 & 6 & -(3n+1)/n \\
 & & \\
-(2n+1)/n & 0 & 5 \\
 & & \\
-1/n & 3 & 0
\end{array} \right); \;\; \; n \in \N.
$$
For $n=1$ we recover the matrix used by Zeeman (1980) for class $7_1$, whose dynamics are illustrated in Figure~\ref{C7_1}. It is easy to check that the Nash equilibria are again the interior Nash equilibrium, $e_1$ and a point $e_{1,3} = \left( (3n+1)/(3n+2), 0, 1/(3n+2)\right)$. This latter Nash equilibrium converges to $e_1$ as $n \rightarrow +\infty$. Otherwise, the dynamics are qualitatively equivalent to those in Figure~\ref{C7_1}. However, the basin of attraction of $e_1$ under RD converges to a set of measure zero. In fact, ${\cal B}_{RD}(e_1)$ is bounded by the edge $[e_1,e_2]$ and by the invariant connection $[e_2 \rightarrow e_{1,3}]$. Because $e_{1,3}$ tends to $e_1$ this bound of ${\cal B}_{RD}(e_1)$ tends to the edge $[e_1,e_2]$. BRD remains qualitatively unchanged. For RD almost all initial conditions lead to the interior Nash equilibrium whereas, for BRD, a set of positive measure of initial conditions chooses to $e_1$ instead, as $n \rightarrow +\infty$.

Recall that the interior Nash equilibrium is stable and $S_1$ is not invariant under RD so that this family does not satisfy the hypotheses in Theorem~\ref{th:basins}.

\section{Concluding remarks}

This article contributes to a better understanding of learning by two classic processes, replicator and best-response, by exposing the importance of the player's indifference among all possible strategies. In particular, even though the initial conditions of play leading to a given pure Nash equilibrium may differ according to whether learning proceeds through replicator or best-response, if the pure Nash equilibrium is locally stable, 
the existence of a fully mixed unstable Nash equilibrium, together with an invariance condition, guarantees that there is always a non-vanishing set from which learning produces the same pure Nash equilibrium as an outcome, regardless of the learning mechanism. The existence of the interior Nash equilibrium means that all types are present in the population mix (or equivalently, that players are indifferent among all possible actions). The instability of this Nash equilibrium and the invariance property of the set $S_i$ containing the pure Nash equilibrium indicate that 
arbitrarily close to total indifference, the type characterising the pure Nash equilibrium is prefered.

Our results also point towards some similarity between the two learning mechanisms and furthermore show the role of the invariance of the indifference sets in maximising this similarity. However, when the basins of attraction differ, their geometry allows room for further research on understanding these differences and their causes. This may also shed some light on the transient behaviour of trajectories, a subject beyond the scope of the present article. Hopefully, experiments will provide further insight.

\paragraph{Compliance with Ethical Standards:}
This study was partly supported by Centro de Matem\'atica da Universidade do Porto (UID/MAT/00144/2013), funded by the Portuguese Government through the Funda\c{c}\~ao para a Ci\^encia e a Tecnologia with national (Minist\'erio da Educa\c{c}\~ao e Ci\^encia) and European structural funds through the programs FEDER, under the partnership agreement PT2020, as well as by a grant from the Reitoria da Universidade do Porto.

The author declares that she has no conflict of interest.

\end{document}